\RequirePackage{fix-cm}
\documentclass[smallextended]{svjour3}       
\smartqed  
\usepackage{graphicx}
\usepackage{amsmath}
\usepackage{amssymb}
\usepackage{amsopn}

\usepackage{ulem}
\usepackage{enumitem}
\DeclareMathOperator{\Sk}{Sk}
 \DeclareMathOperator{\D}{D}
 
 \DeclareMathOperator{\diag}{diag}
 \DeclareMathOperator{\id}{id}

  \DeclareMathOperator{\Mod}{Mod}

\newcommand{\lra}{\longrightarrow}
\newcommand{\ra}{\rightarrow}

\newcommand{\w}{\wedge}

\newcommand{\bsup}{\mathbin{\dot{\vee}}}

\newcommand{\N}{\mathbb{N}}

\newcommand{\Gn}{\mathbf{G}}

\begin{document}

\title{The countable existentially closed pseudocomplemented
semilattice
}


\author{Jo\"el Adler         
}


\institute{J. Adler \at
              P\"adagogische Hochschule Bern, Fabrikstrasse 8, 3012 Bern, Switzerland \\
              Tel.: +41-31-3092452\\
              Fax: +41-31-3092411\\
              \email{joel.adler@phbern.ch}           
}

\date{Received: date / Accepted: date}

\maketitle

\begin{abstract}
As the class $\mathcal{PCSL}$ of pseudocomplemented semilattices is a universal Horn class generated by a single finite structure it has a $\aleph_0$-categorical model companion $\mathcal{PCSL}^*$. As $\mathcal{PCSL}$ is inductive the models of $\mathcal{PCSL}^*$ are exactly the existentially closed models of $\mathcal{PCSL}$. We will construct the unique existentially closed countable model of $\mathcal{PCSL}$ as a direct limit of algebraically closed pseudocomplemented semilattices.
\keywords{Existentially closed \and Pseudocomplemented semilattice \and Model companion \and $\aleph_0$-categoricity}
\subclass{03C05 \and 03C10 \and 06D15}
\end{abstract}

\section{Introduction}\label{sec_introduction}
For a first-order language $\mathcal{L}$ and an $\mathcal{L}$-structure $\mathbf{M}$ with universe $M$ the language $\mathcal{L}(M)$ is obtained by adding a constant symbol for every $m\in M$.
To define the notion of model companion we first have to define the notion of model completeness. An $\mathcal{L}$-theory $T$ is said to be {\it model complete} if for every model $\mathbf{M}$ of $T$
the set of $\mathcal{L}$-sentences $T\cup\diag(\mathbf{M})$ is complete, where $\diag(\mathbf{M})$ is the set of atomic and negated atomic $\mathcal{L}(M)$-sentences that hold in $\mathbf{M}$.
$T^*$ is said to be a {\it model companion} of $T$ if (i) every model of $T^*$ is embeddable in a model of $T$ and vice versa and (ii) $T^*$ is model complete.

A theory $T$ need not have a model companion as is the case for the theory of groups and theory of commutative rings, see Wheeler \cite{Whe}. However, if $T$ is a set of Horn sentences and the
class $\Mod(T)$ of its models is finitely generated then $T$ has a model companion $T^*$ as was shown by Burris and Werner \cite{BuWe}.


If $T$ is additionally inductive---that is $\Mod(T)$ is closed under the union of chains---then we have the characterization $\Mod(T^*)=\Mod(T)^{ec}$, that is, the models of $T^*$ are the
existentially closed models of $T$, see Macintyre \cite{Mac}. A definition of the notions of algebraically and existentially closed can be found in \cite{Ad}. Finally, if $\Mod(T)$ is generated by a single finite structure then $\Mod(T^*)$ is $\aleph_0$-categorical, see Burris \cite{Bu}.

Horn and Balbes \cite{BaHo} proved that $\mathcal{PCSL}$ is equational, Jones  \cite{Jo} showed that it is generated by a single finite
structure. Thus $\mathcal{PCSL}^*$ is $\aleph_0$-categorical and its only countable model is the countable existentially closed pseudocomplemented semilattice.

In Section \ref{sec_preliminaries} we provide the basic properties and algebraic notions concerning pseudocomplemented semilattices (p-semilattices for short), while in Section \ref{sec_construction} the countable existentially closed p-semilattice is constructed.
\section{Basic properties of pseudo\-com\-ple\-men\-ted semi\-lattices and notation}\label{sec_preliminaries}

A p-semilattice $\langle P;\wedge,^*,0\rangle$ is an algebra where $\langle P;\wedge\rangle$ is a meet-semilattice  with least element $0$, and for all $x,y\in P$,
 $x\wedge y = 0$ if and only if $y\leq x^*$ ($x\leq y$ is defined to hold for $x,y\in P$ if $x\wedge y=x$).

From the above definition, it follows that in a p-semilattice the following properties hold:
\begin{align}\label{implication_decreasing}
x\leq y &\implies y^*\leq x^*\\
\label{inequality_double_complement}
x&\leq x^{**}\\
\label{equation_three_stars}
x^*&= x^{***}
\end{align}

$1:=0^*$ is obviously the greatest element of $P$. $x\parallel y$ is defined to hold if neither $x\leq y$ nor $y\leq x$. An element $d$ of $P$ satisfying $d^*=0$ is called {\it dense}, and if
additionally $d\not=1$ holds, then $d$ is called a {\it proper dense} element. For $\mathbf{P}\in\mathcal{PCSL}$ the set $\D(\mathbf{P})$ denotes the subset of dense elements of $\mathbf{P}$, $\langle \D(\mathbf{P});\wedge\rangle$ being a filter of $\langle
P;\wedge\rangle$. An element $s$ is called {\it skeletal} if $s^{**}=s$. The subset of skeletal elements of $\mathbf{P}$ is denoted by $\Sk(\mathbf{P})$. The abuse of notation $\Sk(x)$ for $x\in\Sk(\mathbf{P})$ and $\D(x)$ for $x\in\D(\mathbf{P})$ should not cause ambiguities. Equation \eqref{equation_three_stars} implies $\Sk(\mathbf{P})=\{x^* \colon  x\in P\}$.
In
$\Sk(\mathbf{P})$  the supremum of two elements exists with $\sup_{\Sk}\{a,b\}=(a^*\wedge b^*)^*$ for $a,b\in\Sk(\mathbf{P})$. Instead of $\sup_{\Sk}\{a,b\}$ we use the shorter $a\dot{\vee} b$, assuming $a,b\in\Sk(\mathbf{P})$, which follows from \eqref{implication_decreasing} and \eqref{inequality_double_complement}. Observe that $\langle\Sk(\mathbf{P});\w,\dot{\vee},^*,0,1\rangle$ is a boolean algebra.

\smallskip

For any p-semilattice $\mathbf{P}$ the p-semilattice $\widehat{\mathbf{P}}$ is obtained from $\mathbf{P}$ by adding a new top element. The maximal proper dense element of $\widehat{\mathbf{P}}$ is denoted by $e$. Furthermore, the p-semilattices $\widehat{\mathbf{B}}$ with $\mathbf{B}$ being a boolean algebra interpreted as a p-semilattice are exactly the subdirectly irreducible p-semilattices.
Moreover, let $\mathbf{2}$ denote the two-element boolean
algebra and $\mathbf{A}$ the countable atomfree boolean algebra.

An equational set $\Sigma$ of axioms for $\mathcal{PCSL}$ can be found in \cite{Ad}, for more background on p-semilattices in general consult Frink \cite{Fr} and \cite{Jo}.

\smallskip

%
\smallskip

\smallskip


\smallskip

In Schmid \cite{Sc3} the following characterization of algebraically
closed p-semi\-lattices is established:

\begin{theorem}\label{thm Schmid}
A p-semilattice $\mathbf{P}$ is algebraically closed if and only if for any finite
subalgebra $\mathbf{F}\leq \mathbf{P}$ there exists $r,s\in \N$ and a p-semilattice $\mathbf{F'}$ isomorphic to
$\mathbf{2}^r\times(\widehat{\mathbf{A}})^s$
such that $\mathbf{F}\leq \mathbf{F'}\leq \mathbf{P}$.
\end{theorem}

In \cite{Ad} the following (syntactic) description of existentially closed
p-semi\-lattices is given:

\begin{theorem}\label{thm Adler}
A p-semilattice $\mathbf{P}$ is existentially closed if and only if $\mathbf{P}$ is algebraically
closed and satisfies the following list of axioms:
\setenumerate{label=(EC\theenumi)}
\begin{enumerate}
\item\label{EC1}
\begin{equation*}(\forall b_1,b_2\in \Sk(\mathbf{P}))(\exists
b_3\in \Sk(\mathbf{P}))(b_1<b_2\lra b_1<b_3<b_2),
\end{equation*}
%
\item\label{EC2}
%
\begin{multline*}(\forall b_1,b_2\in \Sk(\mathbf{P}),d\in \D(\mathbf{P}))(\exists
b_3\in \Sk(\mathbf{P}))(\\(b_1\leq b_2<d<1\And b_1^*\parallel d)\\
\lra(b_2<b_3<1\And b_1^*\w b_3\parallel d\And
b_1\,\dot{\vee}\,b_3^*<d)),
\end{multline*}
\item\label{EC3}
\[(\forall b\in\Sk(\mathbf{P})\setminus\{1\})(\exists d\in \D(\mathbf{P}))(b<d\And b^*\parallel d),
\]
\item\label{EC4}
\begin{equation*}
  (\forall d_1,d_2\in \D(\mathbf{P}))(\exists d_3\in P)(d_1<d_2\lra
   (d_1<d_3<d_2)),
\end{equation*}

\item\label{EC5}
\begin{multline*}(\forall b\in \Sk(\mathbf{P}),d_1\in \D(\mathbf{P}))(\exists
d_2\in \D(\mathbf{P}))(0<b<d_1
\\ \lra
(d_2<d_1\And  b\parallel d_2\And d_1\w b^*=d_2\w b^*)).
\end{multline*}
\end{enumerate}
\end{theorem}

Constructing the unique countable model of the model companion $\Sigma^*$ of $\mathcal{PCSL}$ thus amounts to constructing a countable algebraically  closed p-semilattice that satisfies \ref{EC1}--\ref{EC5}.

\section{The construction}\label{sec_construction}

As the objects of the direct limit we are going to construct we take $\{\mathbf{G}_n\colon n\in\N\setminus\{0\}\}$, where $\mathbf{G}_n:=(\widehat{\mathbf{A}})^n$. In view of Theorem \ref{thm Schmid} $\Gn_n$ is algebraically closed for all $n\in\N\setminus\{0\}$. We have to define embeddings $f_n:\Gn_n\ra\Gn_{n+1}$ for $n\geq 1$ such that the direct limit of the directed family $\{\langle
\mathbf{G}_m,g_{m,n}\rangle\colon m,n\in\N, 1\leq m\leq n\}$ where $g_{i,j}:=f_{j-1}\circ\cdots\circ f_{i}$ for
$i<j$ and $g_{i,i}=\id_{G_i}$ additionally satisfies \ref{EC1}--\ref{EC5} of Theorem \ref{thm Adler}. The elements of the direct limit are defined by considering an equivalence relation on $\bigcup_{i=1}^\infty G_i$. $(x_1,\ldots,x_m)$ and $(y_1,\ldots,y_n)$ are in the same equivalence class if $m=n$ or $m<n$ and $g_{m,n}(x_1,\ldots,x_m)=(y_1,\ldots,y_n)$.

We will show that \ref{EC1}--\ref{EC5} hold if
we have the following properties:
\setenumerate{label=(P\theenumi)}
\begin{enumerate}\label{enum1}
\item\label{enum11} For every anti-atom $d$ of $\D(\mathbf{G}_n)$ there is a $k\in\N$ such that $g_{n,n+k}(d)$ is not an anti-atom of $\mathbf{G}_{n+k}$ anymore. This will imply that the order restricted to the dense elements of the direct limit is dense.
  \item\label{enum13} For every $a\in \Sk(\mathbf{G}_n)\setminus\{0,1\}$ there exists $k\in\N$ such that  $\pi_{n+k}(g_{n,n+k}(a))=1$.
\end{enumerate}
To define the embeddings $f_n$ we consider the following functions.
\setenumerate{label=(\theenumi)}
\begin{enumerate}
  \item Let $r\colon\N\setminus\{0\}\to\N\setminus\{0\}$ be such that
  \begin{itemize}\renewcommand{\labelitemi}{$\bullet$}
    \item $r(i)\leq 2i$ for all $i$,
    \item $r^{-1}(j)$ is infinite for all $j$.
  \end{itemize}
  \item Let $s\colon\N\setminus\{0\}\to(\widehat{A}\setminus\{0,e,1\})\times(\N\setminus\{0\})$ be such that
  \begin{itemize}\renewcommand{\labelitemi}{$\bullet$}
    \item $\pi_2((s(i))\leq 2i+1$ for all $i$,
    \item $s^{-1}(a,j)$ is infinite for all $(a,j)$.
  \end{itemize}
    \item Let $U$ be a function with domain $\widehat{A}\setminus\{0,e,1\}$ which assigns to each $\widehat{A}\setminus\{0,e,1\}$ an ultrafilter $U_a$ of $A$ containing $a$.

\end{enumerate}
Now, for each $n\in\N\setminus\{0\}$ we define an embedding $f_n\colon \mathbf{G}_n\to \mathbf{G}_{n+1}$ such that \ref{enum11} and \ref{enum13} hold by

\vspace{0.2cm}
\noindent
$f_n(x_1,\ldots,x_n)=(x_1,\ldots,x_n,x_{r(i)})$ \hfill if $n=2i$\\
$f_n(x_1,\ldots,x_n)=\begin{cases}(x_1,\ldots,x_n,1)&\text{if }x_{\pi_2(s(i))}\in U_{\pi_1(s(i))}\cup \{e\}\\
                        (x_1,\ldots,x_n,0)&\text{otherwise}
                     \end{cases}$
\hfill if $n=2i+1$

\vspace{0.2cm}
For the verification of \ref{enum11} consider an anti-atom $d=(d_1,\ldots,d_n)$ of $\D(\mathbf{G}_n)$. Let $j\in\{1,\ldots,n\}$ be such that $d_j=e$. By the infinity of $r^{-1}(j)$ there is $k>n/2$, $k\in\N$, such that $r(k)=j$. Then, setting $(d_1,\ldots,d_{2k})=g_{n,2k}(d)$,
\begin{align*}
g_{n,2k+1}(d)&=f_{2k}\circ g_{n,2k}(d)\\
&= f_{2k}(d_1,\ldots,d_{2k})\\
&= (d_1,\ldots,d_{2k},d_{r(k)})\\
&= (d_1,\ldots,d_{n+k},d_j)\\
&=(d_1,\ldots,d_{n+k},e)
\end{align*}
follows. Thus, $g_{n,2k+1}(d)$ is not an anti-atom of $\mathbf{G}_{n,2k+1}$ anymore.

For the verification of \ref{enum13} consider $a=(a_1,\ldots,a_n)\in\Sk(\mathbf{G}_n)\setminus\{0,1\}$ and let $j\in\{1,\ldots,n\}$ be such that $0<a_j<1$. By the infinity of $s^{-1}(a_j,j)$ there is $k>n/2$, $k\in\N$, such that $\pi_1(s(k))=a_j$ and $\pi_2(s(k))=j$. Then, setting $(a_1,\ldots,a_{2k+1})=g_{n,2k+1}(a)$, we have
\begin{align*}
g_{n,2k+2}(a)&=f_{2k+1}\circ g_{n,2k+1}(a)\\
&= f_{2k+1}(a_1,\ldots,a_{2k+1})\\
&= (a_1,\ldots,a_{2k+1},1)
\end{align*}
The last equality holds because $a_{\pi_2(s(k))}=a_j\in U_{a_j}=U_{a_{\pi_1(s(k))}}$.
\begin{claim}
The direct limit $\mathbf{G}$ of the directed family $\{\langle
\mathbf{G}_m,g_{m,n}\rangle \colon m,n\in\N, 1\leq m\leq n\}$ of p-semilattices is
countable and existentially closed.
\end{claim}
\begin{proof}
$\mathbf{G}$ is countable since a countable union of countable sets
is countable. That $\mathbf{G}$ is algebraically closed follows from
Theorem \ref{thm Schmid}: Let $\mathbf{S}$ be a finite subalgebra of
$\mathbf{G}$.  By the construction of $\mathbf{G}$, there is
an $n\in\N$ such that that the carrier $S$ of $\mathbf{S}$ is a subset of $G_n$. Therefore, there is a subalgebra $\mathbf{S}'$ of
$\mathbf{G}$ isomorphic to $\mathbf{G}_n=(\widehat{\mathbf{A}})^n$
extending $\mathbf{S}$.

By Theorem \ref{thm Adler} it remains to show that
$\mathbf{G}$ satisfies \ref{EC1}--\ref{EC5}. \ref{EC1}---the order restricted to the skeletal elements is dense---is satisfied
as it is satisfied in $\mathbf{A}$. To prove axioms \ref{EC2}--\ref{EC5} we denote for $x\in\bigcup_{n=1}^{\infty}G_n$ by $[x]\in G$ the equivalence class of $x$.

For \ref{EC2} consider arbitrary $b_1,b_2\in \Sk(\mathbf{G})$ and $d\in
\D(\mathbf{G})$ such that $b_1\leq b_2<d$ and $b_1^*\parallel d$. There is
$n\in\N$ and  $x=(x_1,\ldots,x_n)$, $y=(y_1,\ldots,y_n)$, $w=(w_1,\ldots,w_n)\in G_n$
such that $b_1=[x]$, $b_2=[y]$, $d=[w]$, $\Sk(x)$, $\Sk(y)$, $\D(w)$, $x\leq y<w$ and
$x^*\parallel w$. We first assume that $w$ is not an anti-atom of
$\mathbf{G}_n$. Then, without loss of generality, we can assume
$x_1=0$, $w_1=w_2=e$. Then put $z=(1,z_2,1,\ldots,1)$ with $y_2<z_2<e$
to obtain $y<z$, $z\parallel w$, $x^*\w z\parallel w$ and $x\bsup z^*< w$. The last inequality follows from $x_2\leq y_2<z_2$, which implies $0<x_2^*\wedge z_2$, thus
$e>(x_2^*\wedge z_2)^*=x_2\bsup z_2^*$.
Putting $b_3=[z]$ yields the element requested in \ref{EC2}.

If $w$ is an anti-atom there is by  \ref{enum11} an
$l\in\N$ such that $g_{n,n+l}(w)$ is not an
anti-atom of $\mathbf{G}_{n+l}$ anymore. For
$x':=g_{n,n+l}(x)$,
$y':=g_{n,n+l}(y)$ and $w':=g_{n,n+l}(w)$ we find as above
$z\in G_{n+l}$ such that $y'<z$, $x'^*\w z\parallel w'$ and $x'\bsup z^*< w'$.
Putting $b_3=[z]$ then yields the element requested in \ref{EC2} because $[x]=[x']$, $[y]=[y']$, $[w]=[w']$.

For \ref{EC3} consider an arbitrary $b\in \Sk(\mathbf{G})$ with $b<1$. There is
$n\in\N$ and  $x=(x_1,\ldots,x_n)\in G_n$ such that $b=[x]$, $\Sk(x)$ and $x<1$. By \ref{enum13} there is $k\in\N$ such that $\pi_{n+k}(g_{n,n+k}(x^*))=1$. Then, with $(1,\ldots,1,e)\in G_{n+k}$, we can choose $d=[(1,\ldots,1,e)]$ as the element requested in \ref{EC3}.

For \ref{EC4} consider arbitrary $d_1,d_2\in\D(\mathbf{G})$
such that $d_1<d_2$. There is $n\in\N$ and $x,y\in G_n$
such that $d_1=[x]$, $d_2=[y]$. There are $l\in\N$ and
$z\in\D(\mathbf{G}_{n+l})$ such that $g_{n,n+l}(x)<z<g_{n,n+l}(y)$: We have
$x=\bigwedge_{j\in J_x}a_j$, $y=\bigwedge_{j\in J_y}a_j$ for subsets
$J_y\subsetneq J_x\subseteq\{1,\ldots,n\}$, $a_j$ being an anti-atom of $\D(\mathbf{G}_n)$ for $j\in J_x$. For $j_0\in J_x\setminus J_y$ there is by  \ref{enum11}
a least $l\in\N$ such
that $g_{n,n+l}(a_{j_0})$ is not an anti-atom of
$\mathbf{G}_{n+l}$ anymore, that is, there is are anti-atoms  $u_1,u_2\in G_{n+l}$
with $g_{n,n+l}(a_{j_0})=u_1\wedge u_2$ and $g_{n,n+l}(a_{j})\parallel u_i$ for all $j\in J_y$, $i=1,2$. Because $\D(\mathbf{G}_{n+l})$ is boolean, that is, a $\wedge$-reduct of a boolean algebra, we obtain
\begin{multline*}g_{n,n+l}(x)=\bigwedge_{j\in J_x\setminus\{j_0\}}g_{n,n+l}(a_{j})\wedge u_1\wedge u_2<\bigwedge_{j\in J_x\setminus\{j_0\}}g_{n,n+l}(a_{j})\wedge u_1\\\leq\bigwedge_{j\in J_y}g_{n,n+l}(a_{j})\wedge u_1<\bigwedge_{j\in J_y}g_{n,n+l}(a_{j}),
\end{multline*}
which implies
\begin{equation*}\label{boolean_inequality}
g_{n,n+l}(x)<g_{n,n+l}(y)\wedge u_1<g_{n,n+l}(y).
\end{equation*}
We have
\[d_1=[x]=[g_{n,n+l}(x)]<[g_{n,n+l}(y)\wedge u_1]<[g_{n,n+l}(y)]=[y]=d_2, \]
and we can choose $d_3=[g_{n,n+l}(y)\wedge u_1]$ as the element requested in \ref{EC4}.

For \ref{EC5} consider arbitrary $b\in \Sk(\mathbf{G})$ and $d_1\in
\D(\mathbf{G})$ such that $0<b<d_1$. There is
$n\in\N$ and $x=(x_1,\ldots,x_n)$, $y=(y_1,\ldots,y_n)\in G_n$
such that $b=[x]$, $d_1=[y]$, $\Sk(x)$, $\D(y)$, $0<x<y$.  Let us assume that there is no $z\in\D(\Gn_n)$ such that $z<y$,
$x||z$ and $x^*\w y=x^*\w z$, since otherwise we put $d_2=[z]$.

By  \ref{enum13} there is an $l\in\N$ such that $\pi_{n+l}(g_{n,n+l}(x))=\pi_{n+l}(g_{n,n+l}(y))=1$. Defining $z\in G_{n+l}$ by putting $z_j=\pi_j(g_{n,n+l}(y))$ for $1\leq j\leq n+l-1$ and $z_{n+l}=e$ we then can choose $d_2=[z]$ as the element requested in \ref{EC5}.

\end{proof}




\end{document}